\documentclass{article}
\usepackage[latin1]{inputenc}
\usepackage[english]{babel}
\usepackage{indentfirst}
\usepackage{amstext}
\usepackage{amsthm}
\usepackage{setspace}
\usepackage{amsfonts}
\usepackage{textcomp}
\usepackage{amssymb}
\usepackage{amscd}
\usepackage{epsf}
\usepackage{graphicx}
\usepackage{epsfig}
\usepackage{amsmath}
\RequirePackage{color}

\newcommand {\demo}{\hskip -0.6cm{\bf Proof:  }}
\newcommand {\fim}{\hfill{$\square$}\vskip 1pc}

\newcommand {\R}{\mathbb{R}}
\newcommand {\N}{\mathbb{N}}
\newcommand {\Z}{\mathbb{Z}}

\newcommand {\F}{\mathbb{F}}
\newcommand {\G}{\mathrm{G}}

\newtheorem{theorem}{Theorem}[section]
\newtheorem{lema}[theorem]{Lemma}
\newtheorem{corolario}[theorem]{Corollary}
\newtheorem{definition}[theorem]{Definition}
\newtheorem{proposition}[theorem]{Proposition}
\newtheorem{example}[theorem]{Example}
\newtheorem{obs}[theorem]{Remark}

\begin{document}

\title{Representations of relative Cohn path algebras}

\author{
\small{Crist\'{o}bal Gil Canto, }\\
\footnotesize{Departamento de \'{A}lgebra, Geometr\'{i}a y Topolog\'{i}a}\\
\footnotesize{Universidad de M\'{a}laga}\\
\footnotesize{29071 M\'{a}laga, Spain} \\
\footnotesize{\texttt{cristogilcanto@gmail.com / cgilc@uma.es}}
\and
\small{Daniel Gon\c{c}alves}\\
\footnotesize{Departmento de Matem\'atica}\\
\footnotesize{Universidade Federal de Santa Catarina}\\
\footnotesize{88040-900 Florian\'{o}polis - SC, Brazil}\\
\footnotesize{\texttt{daemig@gmail.com}}}

\date{}

\maketitle

\begin{abstract} We study relative Cohn path algebras, also known as Leavitt-Cohn path algebras, and we realize them as partial skew group rings (to do this we prove uniqueness theorems for relative Cohn path algebras). Furthermore, given any graph $E$ we define $E$-relative branching systems and prove how they induce representations of the associated relative Cohn path algebra. We give necessary and sufficient conditions for faithfulness of the representations associated to $E$-relative branching systems (this improves previous results known to Leavitt path algebras of row-finite graphs with no sinks). To prove this last result we show first a version, for relative Cohn-path algebras, of the reduction theorem for Leavitt path algebras.
\end{abstract}

\vspace{1.0pc}

\vspace{1.0pc}

\section{Introduction}

Leavitt path algebras of directed graphs are considered a mainstream topic of research in algebra. Introduced initially in \cite{AA} and \cite{AMP}, they are the algebraic counterparts of graph $C^*$-algebras and generalizations of Leavitt algebras. Relative Cohn path algebras, as they are named in \cite{AAS}, are generalizations of Leavitt path algebras in which the so-called Cuntz-Krieger relation (CK2) hold just for a subset of regular vertices of the directed graph. Therefore relative Cohn path algebras include Cohn path algebras (the case when we completely omit the aforementioned (CK2) axiom), Leavitt path algebras, and everything in between. They are also known as Cohn-Leavitt path algebras; see for example \cite{KO} and \cite{L1}.

Previously to the consideration of relative Cohn path algebras, its $C^*$-analog was introduced in \cite{MT}, where Muhly and Tomforde presented the relative graph $C^*$-algebra $C^*(E,S)$ of a graph $E$ and $S$ a subset of regular vertices of $E$. Moreover, relative Cuntz-Krieger algebras associated to finitely aligned higher-rank graphs were defined in \cite{Sims}. Ara and Goodearl, in \cite{AG}, introduced and investigated relative Cohn path algebras of separated graphs, a more general class of graphs that encompass relative graph algebras.

Differently from what it might seem at first, the class of relative Cohn path algebras is equal to the class of Leavitt path algebras, see \cite[Theorem 1.5.17]{AAS} (the same is true for relative graph C*-algebras and graph C*-algebras, see \cite{MT}). In any case, working with relative Cohn path algebras is an useful and advantageous tool to unite considerations regarding both Cohn and Leavitt path algebras. For example, in \cite{L2} the use of relative Cohn path algebras is shown to be necessary when representing a Leavitt path algebra $L_K(E)$ as a direct limit of subalgebras of its finite subgraphs (in the case when $E$ fails to be row-finite); here Vas also presents examples which illustrate the benefits of considering relative Cohn path algebras over using Leavitt path algebras alone (similar considerations are made in \cite{MT} for relative graph algebras).
Furthermore, the Gelfand-Kirillov dimension and the invariant basis number (IBN) of relative Cohn path algebras are studied in \cite{MM} and \cite{KO}, respectively. In the context of C*-algebra theory the study of relative algebras is also of great importance. For example, it is not clear if the relative Cuntz-Krieger algebras of finitely aligned higher rank graphs form a larger class then the class of 'non-relative' algebras, see \cite{Sims}. In relation with dynamics and operator theory, relative graph algebras are the key ingredient in the unifying study of KMS states associated to graph algebras, see \cite{toke}. For the latter, the key step is to build a relative boundary path space and to realize relative graph C*-algebras as partial crossed products.  To obtain an algebraic version of this result is one of our goals, as we describe below.

In this paper we connect the theory of relative Cohn path algebras with another key concept arising from operator algebras theory: partial skew group rings. Partial skew group rings were introduced by Exel and Dokuchaev in \cite{Ex} as algebraic analogues of $C^*$ partial crossed products. This notion has been in constant development recently and for example \cite{BG, Misha, dg, GOR, OinertChain} illustrate how the interaction between the theory of partial skew group rings and the theory of non-commutative rings can be fruitful. In our context, given a graph $E$, one could maybe use the equivalence between the class of Leavitt path algebras and the class of relative Cohn path algebra to apply the results in \cite{rg1} and obtain a realization of a relative Cohn path algebra of $E$ as a partial skew group ring. But, this would lead to an action of the free group on the edges of an extended graph and not on the  free group on the edges of $E$. Our approach is more natural, as we define a relative boundary path space, and an action of the free group on the edges of the graph $E$, that give rise to the relative Cohn path algebra. Using our characterization we are able to describe maximal commutativity of a certain abelian subalgebra of the relative Cohn algebra in terms of a combinatorial property of the graph (which we call Relative Condition~(L)). This commutative subalgebra corresponds to the one studied in \cite{GN} for Leavitt path algebras.

The other aspect of our work is to establish a connection between the theory of representations of relative Cohn path algebras with the theory of branching systems. Notice that branching systems, and representations arising from them, have connections with wavelets, C*-algebra theory, ring theory, among other areas (see \cite{MR3404559,FGKP, GLR1, GLR, GLR2, GR4, MR2903145, GR3, MR2848777} for example). In this paper we prove generalized versions of several representation theorems for Leavitt path algebras in \cite{GR}, that is, we show how to obtain representations of relative Cohn path algebras from $E$-relative (algebraic) branching systems and study them. In particular we are able to obtain deeper versions, for Leavitt path algebras, of several results given in \cite{GR}. Concretely, we give necessary and sufficient conditions to guarantee faithfulness of a representation induced by an $E$-relative branching system. This result also extends a key result in \cite{GR4} for the algebras of separated graphs that can be seen as relative Cohn path algebras. Further to improve the known result for Leavitt path algebras, our proof is simpler, and relies on a version of the reduction theorem for relative Cohn path algebras (a result we prove in Section \ref{Sec:Uniqueness}). As an application of our theorem we build faithful representations, arising from $E$-relative branching systems, of any relative Cohn path algebra (this also extends results in \cite{GR4} for row finite graphs with no sinks).

Our work is organized as follows. First in Section \ref{Sec:Prelimin} we recall some basic terminology and definitions about relative Cohn path algebras and partial skew group rings. In Section \ref{Sec:Uniqueness} we prove the uniqueness theorems for relative Cohn path algebras, which we will need in order to realize them as partial skew group rings. Furthermore, we show the so-called reduction theorem for relative Cohn path algebras (Proposition \ref{PropReduction}), in the spirit of the one originally given in \cite{AMMS} for Leavitt path algebras. In Section \ref{section2} we apply the previous results to associate relative Cohn path algebras to partial skew group rings.
Finally, we devote Section \ref{Sec:branching} to the introduction of the notion of $E$-relative (algebraic) branching systems and the representations of the relative Cohn path algebra $C_K^X(E)$ induced by them. %We show that, for any graph $E$, we may always find representations induced by $E$-relative branching systems. Furthermore,
We present one of the main results of the paper, Theorem~\ref{faithfulrep}, where we describe necessary and sufficient conditions for a representation arising from a $E$-relative branching system to be faithful. Furthermore, we construct faithful representations, arising from $E$-relative algebraic branching systems, associated to any relative Cohn path algebra.

\section{Preliminaries}\label{Sec:Prelimin}

Throughout the paper $K$ denotes a field and  $K^{\times}:=K\setminus\{0\}$.

\subsection{Relative Cohn path algebras}
Let $E=(E^0,E^1,r,s)$ be a directed graph. The elements of $E^0$ are called \emph{vertices} and the elements of $E^1$ \emph{edges}. If a vertex $v$ emits no edges, that is, if $s^{-1}(v)$ is empty, then $v$
is called a \emph{sink}. A vertex $v$ is called a \emph{regular vertex} if $s^{-1}(v)$ is a finite non-empty set. The set of regular vertices is denoted by ${\rm Reg}(E)$.

A \emph{path of length $n$} in $E$ is a sequence $\xi_1 \xi_2 \ldots \xi_n$ of edges in $E$ such that $r(\xi_i)=s(\xi_{i+1})$ for $i \in \{1,2,\ldots,n-1\}$. If $\xi$ is a path of length $n$, then we write $|\xi|=n$.  We consider vertices in $E^0$ as paths of length zero. The set of all finite paths of length $n$ is denoted by $E^n$ and we let ${\rm Path}(E)=\cup_{n=0}^{\infty} E^n$.

An \emph{infinite path} in $E$ is an infinite sequence $\xi_1 \xi_2 \ldots$ of edges in $E$ such that $r(\xi_i)=s(\xi_{i+1})$ for $i \in \N$. The set of all infinite paths in $E$ is denoted by $E^\infty$.

As usual, the range and source maps can be extended from $E^1$ to $E^\infty \cup {\rm Path}(E)$ by defining $s(\xi):=s(\xi_1)$ for $\xi=\xi_1\xi_2\ldots \in E^\infty$ or $\xi=\xi_1\ldots\xi_n \in {\rm Path}(E)$,
and $r(\xi):=r(\xi_n)$ for $\xi=\xi_1\ldots\xi_n \in {\rm Path}(E)$.

%A finite path $\eta$ is said to be an \emph{initial subpath} of a (possibly infinite) path $\xi$, if there is a path $\xi'$ such that $r(\eta)=s(\xi')$ and $\xi=\eta\xi'$ hold.

A \emph{closed} path $\alpha=e_1...e_n$ in the graph $E$ is a path such that $r(e_i)=s(e_{i+1})$ and $r(\alpha):=r(e_n)=v=s(e_1)=:s(\alpha)$. The closed path $\alpha$ is called a \emph{cycle} if it does not pass through any of its vertices twice, that is, if $s(e_{i})\neq s(e_{j})$ for every $i\neq j$. An \emph{exit} for a path $\alpha =e_{1}\dots e_{n}$ is an edge $e$ such that $s(e)=s(e_{i})$ for some $i$ and $e\neq e_{i}$. We say that $E$ satisfies \emph{Condition }(L) if every simple closed path in $E$ has an exit, or, equivalently, every cycle in $E$ has an exit.

\begin{definition}\label{Def:relativeCohn}{\rm  Let $E$ be an arbitrary graph. Let $X$ be any subset of ${\rm Reg}(E)$. The \emph{Cohn path algebra of $E$ relative to $X$}, denoted $C_K^X(E)$, is the free $K$-algebra generated by the sets $E^0 \cup E^1 \cup \{e^* \ | \ e \in E^1\}$ with relations:

(V) $vw=\delta_{v,w}v$ for $v,w \in E^0$,

(E1)  $s(e)e=er(e)=e$ for $e \in E^1$,

(E2) $r(e)e^*=e^*s(e)=e^*$ for $e \in E^1$,

(CK1) $e^*f = \delta_{e,f}r(e)$ for $e,f \in E^1$ and,

(XCK2) $v = \sum_{e \in s^{-1}(v)}ee^*$ for every vertex $v \in X$.
}
\end{definition}

We denote $Y:={\rm Reg}(E) \setminus X$ throughout  this paper.

We immediately see that the \emph{Cohn path algebra} $C_K(E)$ corresponds to $C_K(E)=C_K^{\emptyset}(E)$ and the \emph{Leavitt path algebra} $L_K(E)$ to $L_K(E)=C_K^{{\rm Reg}(E)}(E)$.

From the axioms of the Definition \ref{Def:relativeCohn} we have that every element of $C_K^X(E)$ can be represented as a sum of the form $\sum_{i=1}^n k_i \alpha_i\beta_i^*$ for some $n \in \mathbb{N}$, paths $\alpha_i, \beta_i$ such that $r(\alpha_i)=r(\beta_i)$, and $k_i \in K$ for every $i=1, \ldots,n$.

We can define an $K$-linear involution $*$ on $C_K^X(E)$ in the following way: $(\sum_{i=1}^n k_i \alpha_i\beta_i^*)^*=\sum_{i=1}^n k_i \beta_i \alpha_i^*$ for $n \in \mathbb{N}$, paths $\alpha_i, \beta_i$ such that $r(\alpha_i)=r(\beta_i)$, and $k_i \in K$ for every $i=1, \ldots,n$.

Also $C_K^X(E)$ is an unital ring if and only if $E^0$ is finite (where the identity is the sum of elements of $E^0$); it has local units if $E^0$ is not finite (the finite sums of distinct vertices are local units).

Following \cite[Corollary 2.1.5]{AAS} another property is that it is also naturally graded by $\mathbb{Z}$ so that the $n$-component is:
$$C_K^X(E)_n=\left \{\sum_i k_i \alpha_i\beta_i^* \ | \ \alpha_i, \beta_i \text{ are paths, } k_i \in K, |\alpha_i|-|\beta_i|=n \text{ for all } i \right\}.$$

 We finish this subsection recalling that any relative Cohn path algebra $C_K^X(E)$ is isomorphic to the Leavitt path algebra of a graph $E(X)$ which is obtained by adding certain new vertices and edges to $E$.

\begin{definition}{\rm \label{graphex} (as in \cite{AAS} Definition 1.5.16)
Let $E$ be an arbitrary graph, $X$ be a subset of ${\rm Reg}(E)$, and $Y={\rm Reg}(E) \setminus X$. Let $Y'=\{v':v\in Y\}$. For $v\in Y$, and for each edge $e\in r^{-1}(v)$, consider a new symbol $e'$. Define the graph $E(X)$ as follows:
$$E(X)^0=E^0 \cup Y' \text{ and } E(X)^1=E^1 \cup \{e':r(e)\in Y \}.$$ For each $e\in E^1$ let  $r_{E(X)}(e')=r(e)'$
and $s_{E(X)}(e')=s(e)$, and $r_{E(X)}(e)=r(e)$
and $s_{E(X)}(e)=s(e)$.
}
\end{definition}

\begin{theorem}\label{relativetoLPA} (as in  \cite{AAS}, Theorem 1.5.18)
Let $E$ be a graph, $X$ a subset of ${\rm Reg}(E)$, and let $E(X)$ be the graph constructed above (Definition~\ref{graphex}). Then there is an isomorphism $\phi:C_K^X(E)\rightarrow L_K(E(X))$ such that $\phi(v)=v+v'$ if $v\in Y$ and $\phi(v)=v$ otherwise. Furthermore, $\phi(e)=e$ if $r_E(e)\notin Y$ and $\phi(e)=e+e'$ if $r_E(e)\in Y$. Moreover, the inverse of $\phi$ is given by an isomorphism $\psi:L_K(E(X))\rightarrow C_K^X(E)$ such that $\psi(v)=v$ if $v\notin Y$, and $\psi(v)= \sum_{e\in s^{-1}(v)} e e^*$, $\psi(v')=v-\sum_{e\in s^{-1}(v)}ee^*$ if $v\in Y$. Also $\psi(e)=e$ if $r(e)\notin Y$, and $\psi(e)=e \sum_{f\in s^{-1}(v)}ff^*$, $\psi(e')=e (v-\sum_{f\in s^{-1}(v)}ff^*)$ if $r(e)=v \in Y$.
\end{theorem}

\subsection{Partial skew group rings}

For later use, we recall the definitions of a partial action and a partial skew group ring as in \cite{Ex}.

A \emph{partial action} of a group $\G$ on a set $\Omega$ is a pair $\alpha= (\{D_{t}\}_{t\in \G}, \ \{\alpha_{t}\}_{t\in \G})$, where for each $t\in \G$, $D_{t}$ is a subset of $\Omega$ and $\alpha_{t}:D_{t^{-1}} \rightarrow D_{t}$ is a bijection such that $D_{e} = \Omega$, $\alpha_{e}$ is the identity in $\Omega$, $\alpha_{t}(D_{t^{-1}} \cap D_{s})=D_{t} \cap D_{ts}$ and $\alpha_{t}(\alpha_{s}(x))=\alpha_{ts}(x),$ for all $x \in D_{s^{-1}} \cap D_{s^{-1} t^{-1}}.$ In case $\Omega$ is an algebra or a ring then the subsets $D_t$ should also be ideals and the maps $\alpha_t$ should be isomorphisms. Associated to a partial action of a group $G$ in a ring $A$ the \emph{partial skew group ring} $A\rtimes_{\alpha} \G$ is defined as the set of all finite formal sums $\sum_{t \in G} a_t\delta_t$, where, for all $t \in G$, $a_t \in D_t$ and $\delta_t$ are symbols. Addition is defined in the usual way and multiplication is determined by $(a_t\delta_t)(b_s\delta_s) = \alpha_t(\alpha_{-t}(a_t)b_s)\delta_{t+s}$.

\section{Uniqueness theorems for relative Cohn path algebras}\label{Sec:Uniqueness}

In this section we develop the main tools we will use in the next sections. These are also interesting results in their own. We begin with the so-called reduction theorem for relative Cohn path algebras (Proposition \ref{PropReduction}), for which we need an auxiliary result first.

\begin{lema}\label{Lem:corner} Let $c=e_1e_2\cdots e_n$ be a cycle without exits based at a vertex $w$, and denote $\mu_0=w$, $\mu_k=e_1\cdots e_k$ for $1 \leq k < n$ with $s(e_k) \in Y$. Then
$$wC_K^X(E)w= \left \{\sum_{\substack{0\leq i\leq t_i,\\0\leq j \leq t_j,\\ 0 \leq k < n}}
l_{ijk} c^i \mu_k \mu_k^* c^{-j} \  |  \ l_{ijk} \in K; t_i, t_j \in \N \cup \{0\} \right \}.$$
\end{lema}

\begin{proof} Following the same ideas as in the proof of \cite[Lemma 2.2.7]{AAS} we have first that any $\gamma \in {\rm Path}(E)$ such that $s(\gamma)=w$ is of the form $c^m\tau_p$ where $m \in \mathbb{Z}^+$, $\tau_0=w$, $\tau_p=e_1\cdots e_p$ for $1 \leq p < n$, and ${\rm deg}(\gamma)=mn+p$.

Consider $\gamma, \lambda \in {\rm Path}(E)$ with $s(\lambda)=s(\gamma)=w$. Suppose that $s(e_k) \in Y$ and $s(e_{k+1}) \in X,\ldots,s(e_p)\in X$ for $k \leq p$. If ${\rm deg}(\gamma)={\rm deg}(\lambda)$ and $\gamma\lambda^*\neq 0$, we have $\gamma\lambda^* = c^qe_1\cdots e_k e_k^*\cdots e_1^*c^{-q}$. If ${\rm deg}(\gamma) >{\rm deg}(\lambda)$ and $\gamma\lambda^*\neq 0$ then
$\gamma\lambda^* = c^{d+q}e_1\cdots e_k e_k^*\cdots e_1^*c^{-q}$, $d \in \mathbb{N}$. On the other hand, ${\rm deg}(\gamma) <{\rm deg}(\lambda)$ and $\gamma\lambda^*\neq 0$ imply $\gamma\lambda^* = c^{q}e_1\cdots e_k e_k^*\cdots e_1^*c^{-q-d}$, $d \in \mathbb{N}$.

For any $\alpha \in wC_K^X(E)w$, write $\alpha = \sum_{i=1}^r l_i \gamma_i\lambda_i^*$ with $l_i \in K$ and $\gamma_i, \lambda_i \in {\rm Path}(E)$ such that $s(\lambda_i)=s(\gamma_i)=w$ for all $1 \leq i\leq r$. Then, using the computations in the previous paragraph we get the desired result.
\end{proof}

Although we cannot determine whether the corner $wC_K^X(E)w$ given in Lemma \ref{Lem:corner} is isomorphic to some known algebra, we provide below an example of elements in a specific corner.

\begin{example}{\rm
Consider the graph in the picture below with $Y=\{w,v\}$.

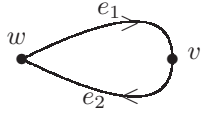
\centerline{
\setlength{\unitlength}{2cm}
\begin{picture}(0,0.5)
\put(0,0){\circle*{0.07}}
\put(-1,0){\circle*{0.07}}
\put(-0.35,0.195){$>$}
\put(-1,0){\qbezier(0,0)(1,0.5)(1,0)}
\put(-1,0){\qbezier(0,0)(1,-0.5)(1,0)}
\put(-0.35,-0.29){$<$}
\put(-0.5,0.3){$e_1$}
\put(-0.6,-0.3){$e_2$}
\put(-1.1,0.1){$w$}
\put(0.1,0){$v$}
\end{picture}
}
\vspace{1cm}
Take the cycle $c=e_1e_2$ based at $w$. Notice that elements in $wC_K^X(E)w$ include $c^i(c^*)^j$ (since $v \in Y$), and $c^ie_1e_1^*(c^*)^j$ (since also $w \in Y$) for $i,j \in \N \cup \{0\}$.
}
\end{example}

\begin{proposition}\label{PropReduction} For any nonzero element $\alpha \in C_K^X(E)$ there exist $\mu, \eta \in {\rm Path}(E)$ such that either:
\begin{itemize}
\item[(i)] $0 \neq \mu^*\alpha\eta = ku$, for some $k \in K^{\times}$ and $u \in E^0$, or
\item[(ii)] $0 \neq \mu^*\alpha\eta = k(v-\sum_{e \in s^{-1}(v)} ee^*)$, for some $k \in K^{\times}$ and $v \in Y$, or
\item[(iii)] $0 \neq \mu^*\alpha\eta \in wC_K^X(E)w$, for some cycle without exits based at a vertex $w$.
\end{itemize}
\end{proposition}

\begin{proof}
We show first that for a nonzero element $\alpha \in C_K^X(E)$, there exist paths $\mu,\eta \in {\rm Path}(E)$ such that
$0 \neq \alpha\eta \in KE$ or $0 \neq \mu^*\alpha\eta = k(v-\sum_{e \in s^{-1}(v)}ee^*)$, for some $k\in K^{\times}$ and $v \in Y$.

Consider a vertex $v\in E^0$ such that $\alpha v\neq 0$. Write $\alpha v= \sum_{i=1}^m \alpha_ie_i^\ast +\alpha'$, with $e_i\in
E^1$, $e_i\neq e_j$ for $i\neq j$ and $\alpha_i, \alpha' \in C_K^X(E)$, $\alpha'$ in only real edges and such that this is a
minimal representation of $\alpha v$ in ghost edges.

If $\alpha ve_i=0$ for every $i\in \{1, \dots, m\}$, then $0=\alpha ve_i= \alpha_i +\alpha' e_i$. Hence $\alpha_i=-\alpha' e_i$, and $\alpha
v=\sum_{i=1}^m -\alpha' e_i e_i^\ast +\alpha'=\alpha' (\sum_{i=1}^m - e_i e_i^\ast +v)\neq 0$. This implies that
$\sum_{i=1}^m - e_i e_i^\ast +v\neq 0$. There are now two cases, depending whether $v\in Y$ or not. Suppose first that $v \in E^0 \setminus Y$. Since $s(e_i)=v$ for every $i$, this means that there exists $f\in E^1$, $f\neq e_i$ for every $i$, with $s(f)=v$. In this case, $\alpha v f =\alpha' f \neq 0$ (because $\alpha '$ is in only real edges), with $\alpha' f$ in only real edges, so we conclude. In the second case, assume that $v \in Y$. Then, multiplying the equation $\alpha
v=\alpha' (\sum_{i=1}^m - e_i e_i^\ast +v)\neq 0$ by $\alpha'^*$, we get that $\alpha'^*\alpha
v=\sum_{i=1}^m - e_i e_i^\ast +v\neq 0$, and we obtain the desired result.

Continue with the case $\alpha ve_i\neq 0$ for some $i$, say for $i=1$. Then $0\neq \alpha v e_1=\alpha_1+\alpha' e_1$, with $\alpha_1+\alpha' e_1$ having
strictly less degree in ghost edges than $\alpha$. Repeating the argument above a finite number of steps we prove our first statement.

\smallskip

Now, consider  $0 \neq \alpha \in C_K^X(E)$. Suppose that there exists a path $\eta \in {\rm Path}(E)$ such that
$\beta:= \alpha\eta \in KE \setminus \{0\}$ (if not, by what is proved above, we are finished). Write $0\ne \beta=\sum_{i=1}^r k_i\beta_i$ as
a linear combination of different paths $\beta_i$ with $k_i\ne 0$ for any $i$. We prove by induction on $r$
that, after multiplying $\beta$ on the left and/or the right, we get a vertex or a element in $wC_K^X(E)w$ for some cycle without exits based at a vertex $w$.

 For $r=1$, if $\beta_1$ has degree $0$ then it is a vertex and we are finished. Otherwise we have
$\beta=k_1\beta_1=k_1 f_1\cdots f_n$, so that $k_1^{-1}f_n^*\cdots f_1^*\beta=v$ where $v=r(f_n)\in E^0$.

Suppose by induction that the property is true for any nonzero element which is a sum of less than $r$ paths in the
conditions above. Write $0\ne \beta=\sum_{i=1}^r k_i\beta_i$ such that $\deg(\beta_i)\le\deg(\beta_{i+1})$ for
any $i$.  If for some $i$ we have $\deg( \beta_i)=\deg(\beta_{i+1})$ then, since $\beta_i\ne\beta_{i+1}$,
there is some path $\mu$ such that $\beta_i=\mu f\nu$ and $\beta_{i+1}=\mu f'\nu'$ where $f,f'\in E^1$ are
different and $\nu,\nu'$ are paths. Thus $0\ne f^*\mu^*\beta$ and we can apply the induction hypothesis to this
element. So we can go on supposing that $\deg(\beta_i)<\deg(\beta_{i+1})$ for each $i$.

We have  $0\ne \beta_1^*\beta=k_1v+\sum_i k_i\gamma_i$, where $v=r(\beta_1)$ and $\gamma_i=\beta_1^*\beta_i$. If some $\gamma_i$ is null then we apply the
induction hypothesis to $\beta_1^*\beta$ and we are done. Otherwise if some $\gamma_i$ does not start (or finish) in $v$ we apply the induction
hypothesis to $v\beta_1^*\beta \ne 0$ (or $\beta_1^*\beta v\ne 0$). Thus we have $$0\ne z:=\beta_1^* \beta=k_1v+\sum_{i=2}^r k_i\gamma_i,$$ where
$0<\deg(\gamma_2)<\cdots< \deg(\gamma_r)$ and all the paths $\gamma_i$ start and finish in $v$.

If $T(v) \cap P_c(E) = \emptyset$ then, by \cite[Lemma 2.2.8]{AAS},there exists a path $\tau$ such that $\tau^*\beta_1^*\beta\tau = \tau^*z\tau = k_1 r(\tau)$ and we are done.

If $T(v) \cap P_c(E) \neq \emptyset$ then there is a path $\rho$ starting at $v$ such that $w=r(\rho)$ is a vertex in a cycle without exits. In this case $0 \neq \rho^*\beta_1^*\beta\rho = \rho^*z\rho \in wC_K^X(E)w$ and the proof is complete.
\end{proof}

\begin{obs}{\rm Notice that the proposition above does not follow immediately from the Reduction Theorem (\cite{AAS}, Theorem 2.2.11) and Theorem~\ref{relativetoLPA}. In fact, if we take $v \in L_K(E(X))$ such that $v \in Y$ then it is already ``reduced'' to a vertex; but if we now apply the isomorphism $\psi:L_K(E(X))\rightarrow C_K^X(E)$ then $\psi(v)= \sum_{e\in s^{-1}(v)} e e^*$, which is not in any ``reduced form'' of Proposition~\ref{PropReduction}.
}
\end{obs}

Notice that, by Lemma \ref{Lem:corner}, for $X={\rm Reg}(E)$ (that is $C_K^X(E)=L_K(E)$) we obtain the following well-known result: if $c$ is a cycle without exits based at a vertex $w$ then
$$wL_K(E)w = \left \{ \sum_{r=m}^n l_r c^r \ | \ l_r \in K, m,n \in \mathbb{Z} \right \} \cong K[x,x^{-1}].$$
In particular from Proposition \ref{PropReduction} we get the so-called reduction theorem for Leavitt path algebras:

\begin{corolario}\label{Cor:ReductionTheorem} (as in  \cite{AAS}, Theorem 2.2.11) Let $E$ be an arbitrary graph and $K$ any field. For every nonzero element $\alpha \in L_K(E)$ there exist $\mu, \eta \in {\rm Path}(E)$ such that either:
\begin{itemize}
\item[(i)] $0 \neq \mu^*\alpha\eta = ku$, for some $k \in K^{\times}$ and $u \in E^0$, or
\item[(ii)] $0 \neq \mu^*\alpha\eta=p(c)$, for some cycle without exits $c$ and $p(x)$ a nonzero polynomial in $K[x,x^{-1}]$.
\end{itemize}
\end{corolario}

Another immediate consequence from Proposition \ref{PropReduction} is the following result, which we will use in the proof of the Graded Uniqueness Theorem.

\begin{corolario}\label{CorReductionHomogeneous} Let $\alpha$ be a nonzero homogeneous element of $C_K^X(E)$. Then there exist $\mu, \eta \in {\rm Path}(E)$ such that either
\begin{itemize}
\item[(i)] $0 \neq \mu^*\alpha\eta = ku$, for some $k \in K^{\times}$ and $u \in E^0$, or
\item[(ii)] $0 \neq \mu^*\alpha\eta = k(v-\sum_{e \in s^{-1}(v)} ee^*)$, for some $k \in K^{\times}$ and $v \in Y$.
\end{itemize}
In particular, every nonzero graded ideal of $C_K^X(E)$ contains a vertex or a element of the form $v-\sum_{e \in s^{-1}(v)} ee^*$ for some $v \in Y$.
\end{corolario}

\begin{proof} By the first part of the proof of Proposition \ref{PropReduction} we have that for a nonzero element $\alpha \in C_K^X(E)$, there exist paths $\mu,\eta \in {\rm Path}(E)$ such that $0 \neq \alpha\eta \in KE$ or $0 \neq \mu^*\alpha\eta = k(v-\sum_{e \in s^{-1}(v)}ee^*)$, for some $k\in K^{\times}$ and $v \in Y$. For the second case it is done. Suppose we are in the first case. Since $\alpha$ is a homogeneous element, $0 \neq \alpha\eta$ is a homogeneous element in $KE$. Now we write $\alpha\eta=\sum_{i=1}^r k_i\beta_i$ with $k_i \in K^{\times}$, $\beta_i \neq \beta_j$ and $|\beta_i|=|\beta_j|$ for all $i \neq j$. Therefore $\beta_1^*\alpha\eta=k_1r(\beta_1)$ and we complete the proof.

The particular statement follows immediately.
\end{proof}

We are now in position to show the uniqueness theorems for relative Cohn path algebras.

\begin{theorem}{\rm \textbf{(The Graded Uniqueness Theorem)}}\label{gradeduniqueness}  Consider $A$ a $\mathbb{Z}$-graded ring and $\pi: C_K^X(E) \rightarrow A$ a graded ring homomorphism. Suppose that $\pi(u)\neq 0$ for every vertex $u \in E^0$ and $\pi(v-\sum_{e \in s^{-1}(v)} ee^*)\neq 0$ for every vertex $v \in Y$. Then $\pi$ is injective.
\end{theorem}

\begin{proof} Notice that ${\rm Ker}(\pi)$ is a graded ideal of $C_K^X(E)$. Then by Corollary~\ref{CorReductionHomogeneous}, ${\rm Ker}(\pi)$ is either $\{0\}$, or contains a vertex, or contains a element of the form $v-\sum_{e \in s^{-1}(v)} ee^*$ for some $v \in Y$. By the hypothesis, the only option is ${\rm Ker}(\pi)=\{0\}$.
\end{proof}

\begin{theorem}{\rm \textbf{(The Cuntz-Krieger Uniqueness Theorem)}} Consider $\pi: C_K^X(E) \rightarrow A$ a ring homomorphism. Suppose that the graph $E$ satisfies Condition (L), that $\pi(u)\neq 0$ for every vertex $u \in E^0$, and $\pi(v-\sum_{e \in s^{-1}(v)} ee^*)\neq 0$ for every vertex $v \in Y$. Then $\pi$ is injective.
\end{theorem}

\begin{proof} We use that ${\rm Ker}(\pi)$ is an ideal of $C_K^X(E)$. Let $\alpha$ be a nonzero element in ${\rm Ker}(\pi)$. Since $E$ satisfies Condition (L) then, by Proposition~\ref{PropReduction}, there exist $\mu, \eta \in {\rm Path}(E)$ such that either $0 \neq \mu^*\alpha\eta = ku$, for some $k \in K^{\times}$ and $u \in E^0$, or $0 \neq \mu^*\alpha\eta = k(v-\sum_{e \in s^{-1}(v)} ee^*)$, for some $k \in K^{\times}$ and $v \in Y$. Therefore ${\rm Ker}(\pi)$ either contains a vertex, or contains a element of the form $v-\sum_{e \in s^{-1}(v)} ee^*$ for some $v \in Y$, what contradicts the hypothesis of the theorem. Therefore ${\rm Ker}(\pi)=\{0\}$.
\end{proof}

As a consequence of the isomorphism between relative Cohn path algebras and Leavitt path algebras we obtain other uniqueness theorem. For this aim we previously define:

\begin{definition}{\rm Let $E$ be a graph. We say that $E$ satisfies \emph{Relative Condition (L)} if every cycle $c=e_1\cdots e_n$ such that $s(e_i) \notin Y$ for every $i=1,\ldots,n$, has an exit.
}
\end{definition}

\begin{theorem}{\rm \textbf{(The Relative Cuntz-Krieger Uniqueness Theorem)}} Consider $\pi: C_K^X(E) \rightarrow A$ a ring homomorphism. Suppose that the graph $E$ satisfies Relative Condition (L) and that:
 \begin{itemize}
 \item[(i)] $\pi(u)\neq 0$ for every vertex $u \notin Y$,
 \item[(ii)] $\pi(v-\sum_{e \in s^{-1}(v)} ee^*)\neq 0$ for every vertex $v \in Y$, and
 \item[(iii)] $\pi(\sum_{e \in s^{-1}(v)} ee^*)\neq 0$ for every vertex $v \in Y$.
 \end{itemize}
  Then $\pi$ is injective.
\end{theorem}

\begin{proof} It is straightforward that the corresponding graph $E(X)$ satisfies Condition (L). Consider the isomorphism $\psi: L_K(E(X)) \rightarrow C_K^X(E)$ given in Theorem \ref{relativetoLPA}. We have that $\pi \circ \psi: L_K(E(X)) \rightarrow A$ is a ring homomorphism. Moreover, if $u \notin Y$ then $\pi \circ \psi(u)= \pi(\psi(u))= \pi(u) \neq 0$; if $v \in Y$ then $\pi \circ \psi(v)= \pi(\psi(v))= \pi(\sum_{e \in s^{-1}(v)} ee^*) \neq 0$ and $\pi \circ \psi(v')= \pi(\psi(v'))= \pi(v-\sum_{e \in s^{-1}(v)} ee^*) \neq 0$. By the Cuntz-Krieger Uniqueness Theorem (\cite[Theorem 2.2.16]{AAS}) we have that $\pi \circ \psi$ is injective. Therefore $\pi$ is injective.
\end{proof}

\section{Relative Cohn path algebras as partial skew group rings} \label{section2}

Given a graph $E$ and a subset $X$ of $\text{Reg}(E)$, in this section we describe the associated relative Cohn path algebra as the partial skew group ring associated to a partial action, of the free group on the edges of the graph, on the ``relative algebraic boundary path space". We use this characterization to relate dynamical properties of the action with combinatorial properties of the graph.

Since each relative Cohn path algebra is isomorphic to a Leavitt path algebra (see Theorem~\ref{relativetoLPA}), and in \cite{rg1,GR48} each Leavitt path algebra was realized as a partial skew group ring, we obtain a characterization of relative Cohn path algebras as partial skew group rings by composing isomorphisms. However, the partial skew group ring obtained this way is formed by a partial action of the free group on edges of $E(X)$, so that some ``unnatural" edges $e'$ appear on the free group. Furthermore, the boundary path space will also include paths containing edges of $E(X)$. We believe that the picture we present is the correct one for dealing with relative graphs, as we have an action of the free group on the edges and a ``relative boundary path space" that only involves edges and vertices of $E$. In fact, the ``relative algebraic boundary path space" of a graph is the space of infinite paths in the graph union with finite paths ending in a sink or in a vertex in $Y$. The precise definition follows below (compare it with the analytical counterpart, see \cite{toke}).

\begin{definition}{\rm
Let $E$ be a graph and $X$ be a subset of $\text{Reg}(E)$. Recall that $Y=\text{Reg}(E)\setminus X$. The \emph{algebraic relative boundary path space} $\partial_X E$ is defined by
$$ \partial_X(E) = E^\infty \cup \{\xi \in {\rm Path}(E):r(\xi)\text{ is a sink } \} \cup \{\xi \in {\rm Path}(E):r(\xi)\in Y \}.$$
}
\end{definition}

We denote by $\F$ the free group generated by the set $E^1$. Next we define a partial action $(\{U_c\}_{c\in \F},\{\theta_c\}_{c\in \F})$ of $\F$ on $\partial_X(E)$. The sets and maps are defined as follows:

Let $W=\cup_{n\in \N}E^n$ be the set of all finite paths of length greater or equal to one (as a subset of $\F$), and define:

\begin{itemize}
\item $U_0:=\partial_X(E)$, where $0$ is the neutral element of $\F$.

\item $U_{b^{-1}}:=\{\xi\in \partial_X(E): s(\xi)=r(b)\},$ for all $b\in W$.

\item $U_a:=\{\xi\in \partial_X(E): \xi_1\xi_2...\xi_{|a|}=a\},$ for all $a\in W$.

\item $U_{ab^{-1}}:=U_a,$ for $ab^{-1}\in \F$ with $a,b\in W$, $r(a)=r(b)$ and $ab^{-1}$ in its reduced form (that is, $a_{|a|}\neq b_{|b|}$).

\item $U_c:=\emptyset$, for all other $c \in \F$.

\end{itemize}

Furthermore, let $$U_v=\{\xi\in \partial_X(E):s(\xi)=v\}, \text{ for all } v\in E^0.$$

\begin{obs}{\rm \label{need for injectivity} Note that $v\in U_v$ if, and only if, $v$ is a sink or $v\in Y$.  Moreover, if $v=r(b)$ is a sink then $U_{b^{-1}}=\{r(b)\}$ and $U_b=\{b\}$. Notice also that if $v\in X$ then $U_v=\displaystyle \bigcup_{s(a)=v} U_{a}$.
}
\end{obs}

Next we define the maps $\theta_c:U_{c^{-1}}\rightarrow U_c$. Let $\theta_{0}:U_{0}\rightarrow U_{0}$ be the identity map. For $b\in W$, let  $\theta_b:U_{b^{-1}}\rightarrow U_b$ be the ``add b" or ``creation" map, that is, if $\xi\in U_{b^{-1}}$ then $\theta_b(\xi)=b\xi$ (we are assuming here that $br(b)=b$). % Notice that if $\xi\in X_{b^{-1}}$ is a vertex then $s(\xi)=r(b)$ and hence $b\xi\in X_b$ (we are assuming here that $br(b)=b$).
The inverse of $\theta_b$ is given by the ``erase b" map, that is, $\theta_{b^{-1}}:U_b\rightarrow U_{b^{-1}}$ is given by
$\theta_{b^{-1}}(\eta)= \eta_{|b|+1}\eta_{|b|+2}...$ if $r(b)$ is not a sink and $r(b)\notin Y$, and $\theta_{b^{-1}}(b)=r(b)$, if $r(b)$ is a sink or $r(b)\in Y$. Finally, for $a,b\in W$ with $r(a)=r(b)$ and $a_{|a|}\neq b_{|b|}$ we define $\theta_{ab^{-1}}:U_{ba^{-1}}\rightarrow U_{ab^{-1}}$ as the ``erase b and add a" map, that is, $\theta_{ab^{-1}}(\xi)=a\xi_{(|b|+1)}\xi_{(|b|+2)}...$.

\begin{example}{\rm
Consider the graph below and take $Y=\{r(f_1)\}$.

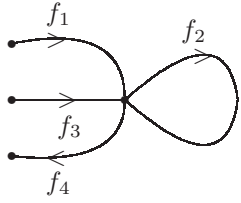
\centerline{
\setlength{\unitlength}{1.5cm}
\begin{picture}(0,1)
\put(0,0){\circle*{0.08}}
\put(-1,0.5){\circle*{0.08}}
\put(-1,-0.5){\circle*{0.08}}
\put(-1,0){\circle*{0.08}}
\put(-1,0){\line(1,0){1}}
\put(-0.6,-0.065){$>$}
\put(-0.6,-0.3){$f_3$}
\put(0,0){\qbezier(0,0)(0.9,0.8)(1,0)}
\put(0,0){\qbezier(0,0)(0.9,-0.8)(1,0)}
\put(0.6,0.33){$>$}
\put(-1,0.5){\qbezier(0,0)(1,0.2)(1,-0.5)}
\put(-1,-0.5){\qbezier(0,0)(1,-0.1)(1,0.5)}
\put(-0.7,0.47){$>$}
\put(-0.7,-0.575){$<$}
\put(-0.7,0.7){$f_1$}
\put(-0.7,-0.8){$f_4$}
\put(0.5,0.55){$f_2$}
\end{picture}}
\vspace{1cm}
We have, for example, $$U_{f_2(f_3)^{-1}}=U_{f_2}=\{f_2, f_2 f_4, f_2f_2, f_2f_2f_4\cdots f_2f_2f_2\cdots\}.$$ Also, since $r(f_1)\in Y$ we have that $r(f_1)\in U_{r(f_1)}$ and hence $U_{r(f_1)}$ contains, but is not equal to $U_{f_2}\cup U_{f_4}$.
}
\end{example}

The set partial action defined above induces a partial action in the algebra level (for more details about the relations between partial actions on sets and partial actions of algebras see \cite{Vivi} and \cite{dg}). For each $c\in \F$, with $U_c\neq\emptyset$, let $F(U_c)$ be the $K$-algebra of functions from $U_c$ to $K$. Note that $F(U_c)$ may be identified with the subset of the functions in $F(\partial_X E)$ that vanishes outside of $U_c$. Furthermore, each $F(U_c)$ is an ideal of the $K$-algebra $F(\partial_X E)$. Now, for each $c\in\F$, define $\alpha_c:F(U_{c^{-1}})\rightarrow F(U_c) $ by $\alpha(f) = f\circ \theta_{c^{-1}}$, which is an $K$-isomorphism. One can now check that the family $\{\{\alpha_c\}_{c\in \F}, \{F(U_c)\}_{c\in \F}\}$ is a partial action of $\F$ on $F(\partial_X E)$.

To obtain the relative Cohn path algebra we need to consider the following restriction of the above partial action:
for each $c\in \F$, and for each $v\in E^0$, define the characteristic maps $1_c:=\chi_{U_c}$ and $1_v:=\chi_{U_v}$. Finally, let $$D(\partial_X(E))=D_0={\rm span}\{\{1_p:p\in \F\setminus\{0\}\}\cup\{1_v:v\in E^0 \}\},$$  (where span means the $K$-linear span) and, for each $p\in \F\setminus\{0\}$, let $D_p\subseteq F(U_p)$ be defined as $1_p D_0$, that is, $$D_p={\rm span}\{\{1_p1_q:q\in \F\}\}.$$
Since $\alpha_p(1_{p^{-1}}1_q)=1_p1_{pq}$ (see \cite{rg1}), consider, for each $p\in \F$, the restriction of $\alpha_p$ to $D_{p^{-1}}$. Notice that $\alpha_{p}:D_{p^{-1}}\rightarrow D_p$ is an isomorphism of $K$-algebras and, furthermore, $\{\{\alpha_p\}_{p\in \F}, \{D_p\}_{p\in \F}\}$ is a partial action. Denote by $D(\partial_X(E))\rtimes_\alpha\F$ the partial skew group ring associated to it.

\begin{theorem} Let $E$ be a graph and $X$ be any subset of ${\rm Reg}(E)$. There exists a $K$-algebra isomorphism $\varphi$, from $C_K^X(E)$ onto $D(\partial_X(E))\rtimes_\alpha\F,$ such that $\varphi(e)=1_e\delta_e$, $\varphi(e^*)=1_{e^{-1}}\delta_{e^{-1}}$, for all $e\in E^1$, and $\varphi(v)=1_v\delta_0$, for all $v\in E^0$.
\end{theorem}

\begin{proof} Consider the sets $\{1_e \delta_e, 1_{e^{-1}} \delta_{e^{-1}}: e\in E^1\}$ and $\{1_v\delta_0: v\in E^0\}$ in $ D(\partial_X(E))\rtimes_\alpha \F$. Proceeding as in \cite{rg1} one can check that these sets satisfy the relations defining the relative Cohn path algebra and hence, by the universal property of $C_K^X(E)$, we obtain the desired homomorphism $\varphi: C_K^X(E) \rightarrow D(\partial_X(E))\rtimes_\alpha\F$, such that, for all $e \in E^1$ and all $v\in E^0$, $\varphi(e)=1_e\delta_e$, $\varphi(e^*)=1_{e^{-1}}\delta_{e^{-1}}$ and $\varphi(v)=1_v\delta_0$.

To show that $\varphi$ is injective we will use the Graded Uniqueness Theorem (Theorem~\ref{gradeduniqueness}). So we need to define a $\Z$-grading in $D(\partial_X(E))\rtimes_\alpha \F$. This is done as in \cite{rg1} and hence we just sketch the steps:
For each $p\in \F$, let $|p|:=m-n$, where $m$ is the number of generators (elements of $E^1$) of $p$ and $n$ is the number of inverses of generators of $p$. Define, for each $z\in \Z$, $A_z\subseteq D(\partial_X(E))\rtimes_\alpha\F$ as the $K$-linear span of $\{a_p\delta_p:a_p\in D_p \text{ and } |p|=z\}$. Then $\{A_z\}_{z\in \Z}$ is a $\Z$-grading of $D(\partial_X(E))\rtimes_\alpha\F$. Remember $C_K^X(E)$ is a $\Z$ graded $K-$algebra, with the grading induced by the length of the paths. Now, for each  $ab^*\in C_K^X(E)$ with $|a|-|b|=z$, we have that $\varphi(ab^*)\in D_{ab^{-1}}\delta_{ab^{-1}}$. Since $|ab^{-1}|=|a|-|b|=z$ then $D_{ab^{-1}}\delta_{ab^{-1}}\subseteq A_z$, and hence $\varphi$ is a $\Z$ graded isomorphism.

To apply the Graded Uniqueness Theorem (Theorem~\ref{gradeduniqueness}), we still need to check that $\varphi(v) \neq 0$ for all $v\in E^0$ (which is straightforward since $\varphi(v)= 1_v \delta_0$ and $U_v\neq \emptyset$ for all $v\in E^0$), and that $\pi(v-\sum_{e \in s^{-1}(v)} ee^*)\neq 0$ for every vertex $v \in Y$. Notice that if $v \in Y$ then, by Remark~\ref{need for injectivity}, $v \in U_v$. On the other hand $v\notin U_e$ for any $e\in s^{-1}(v)$ (since any element in $U_e$ has length at least one). Therefore $1_v \neq \sum_{e\in s^{-1}(v)} 1_e$ and hence $\pi(v-\sum_{e \in s^{-1}(v)} ee^*) = 1_v - \sum_{e\in s^{-1}(v)} 1_e \neq 0$. We conclude that $\varphi$ is injective.

The proof that $\varphi$ is surjective is identical to the proof given in \cite[Theorem~3.3]{rg1}.
\end{proof}

The interplay between combinatorial and algebraic objects is a driving force in the study of Leavitt path algebras and in other areas of Mathematics (see \cite{KK} for an example of this interplay out of Leavitt path algebras theory). Usually to make the connection between a combinatorial property of a graph (for example), and an algebra associated to it, one builds an intermediate dynamical system with properties that model the algebraic and combinatorial aspects under study. This is the case at hand. In our setting, the combinatorial object is composed by a graph and a subset $X$ of the regular vertices, and the algebra associated to it is the relative Cohn path algebra.

Given a graph $E$, and a non-empty subset $Y$, the associated relative Cohn path algebra is never simple, since the graph $E(X)$ has sinks (which imply the existence of hereditary and saturated sets). In \cite{BGOR, GOR} simplicity of a partial skew ring $A\rtimes \F$ was characterized in terms of maximal commutativity of $A$ and $\F-$simplicity of $A$. Below we show that maximal commutativity of $D(\partial_X(E))$ in $D(\partial_X(E))\rtimes_\alpha\F,$ is equivalent to Relative Condition~(L) in the graph. Therefore, for any relative graph with $Y\neq \emptyset$ and that satisfies the
Relative Condition~(L), $D(\partial_X(E))$ is never $\F-$simple.

\begin{proposition} Let $E$ be a graph and $X$ be a subset of ${\rm Reg}(E)$. Then $D(\partial_X(E))$ is maximal commutative in $D(\partial_X(E))\rtimes_\alpha\F$ if, and only if, the graph E satisfies Relative Condition~(L).
\end{proposition}

\begin{proof}Suppose first that $E$ satisfies Relative Condition~ $(L)$. Let $a_t\in D_t$, with $t\neq 0$ and $a_t\neq 0$, be such that $a_t\delta_t\cdot a_0\delta_0=a_0\delta_0\cdot a_t\delta_t$ for each $a_0\in D_0$, that is, such that \begin{equation}\label{eq1} \alpha_t(\alpha_{t^{-1}}(a_t)a_0)=a_ta_0 \end{equation} for all $a_0\in D_0$.

Taking  $a_0=1_{t^{-1}}$ in Equation~\eqref{eq1} we obtain that $a_t=a_t1_{t^{-1}}$ and hence the support of $a_t$ is contained in $U_t \cap U_{t^{-1}}$. So either $t\in W$ or $t=r^{-1}$ with $r\in W$. If $t\in W$ then $t$ is a closed path and if $t=r^{-1}$ then $r$ is a closed path. Furthermore, by induction we obtain that $a_t=a_t1_{(t^n)^{-1}}$ and $a_t1_{t^n}=a_t$, for all $n\in \N$.

Let $\xi\in \text{supp} (a_t)$, that is, $a_t(\xi)\neq 0$. Notice that, since $a_t\in D_t$, there exists an $M$ such that for each $\mu\in U_t$ with $\mu_1\cdots \mu_M=\xi_1\cdots \xi_M$ it holds that $a_t(\mu)=a_t(\xi)$.

Suppose that $t\in W$. If $t$ is a closed path such that the source of each edge in $t$ belongs to $X$ then, by the Relative Condition~(L), $t$ has an exit. The proof now follows as the proof of Proposition~3.1 in \cite{GOR}. So suppose $t$ is a closed path, say $t=t_1\ldots t_k$, and $s(t_j)\in Y$ for some $j$. Let $\mu=\xi_1 \ldots \xi_M \mu_{M+1} \ldots \mu_{L}$ (a finite path), where $r(\mu_L) = s(t_j)$. Then we can find an $n\in \N$ such that $1_{t^n}(\mu)=0$ and this implies that $0\neq a_t(\mu)= a_t 1_{t^n}(\mu)=0$, a contradiction. The case $t=r^{-1}$, with $r$ a closed path, is done analogously.

We conclude that there is no $a_t \in D_t$, with $t\neq 0$, such that $a_t\delta_t$ commutes with each element of $D_0\delta_0$. Hence $D(\partial_X(E))\delta_0$ is maximal commutative.

Suppose now that $E$ does not satisfy Relative Condition~$(L)$, that is, there exists a closed path $t=t_1...t_m$, such that $s(t_i)\notin Y$ for all $i$, which has no exit. Then, proceeding as in the proof of  Proposition~3.1 in \cite{GOR} we get that $1_t\delta_t$ commutes with  $D(\partial_X(E)) \delta_0$ and so  $D(\partial_X(E))$ is not maximal commutative.
\end{proof}

\section{Representations of $C_K^X(E)$ arising from relative branching systems}\label{Sec:branching}

In this section we define $E$-relative algebraic branching systems associated to a directed graph $E$ (and a subset $X$ of $\text{Reg}(E)$) and study the representations of relative Cohn path algebras associated to such systems.

We start with the definition of an $E$-relative algebraic branching system (this is motivated by definitions in \cite{GR}).

\begin{definition}{\rm \label{brancsystem}
Let $E$ be a graph and $X$ a subset of ${\rm Reg}(E)$. Let $\mathfrak{X}$ be a set and let $\{R_e\}_{e\in E^1}$, $\{D_v\}_{v\in E^0}$ be families of subsets of $\mathfrak{X}$ such that:
\begin{enumerate}
\item $R_e\cap R_d= \emptyset$ for each $d,e\in E^1$ with $d\neq e$;
\item $D_u\cap D_v= \emptyset$ for each $u,v\in E^0$ with $u\neq v$;
\item $R_e\subseteq D_{s(e)}$ for each $e\in E^1$;
\item $D_v=\bigcup\limits_{e:s(e)=v}R_e$\,\,\,\,\, for each $v \in X$; and
\item for each $e\in E^1$, there exists a bijective map $f_e:D_{r(e)}\rightarrow R_e$.
\end{enumerate}

A set $\mathfrak{X}$, with families of subsets $\{R_e\}_{e\in E^1}$, $\{D_v\}_{v\in E^0}$, and maps $f_e$ as above, is called an \emph{$E$- relative (algebraic) branching system}, and we denote it by $(\mathfrak{X},\{R_e\}_{e\in E^1}, \{D_v\}_{v\in E^0}, \{f_e\}_{e\in E^1})$, or when no confusion arises, simply by $\mathfrak{X}$.
}
\end{definition}

We show that each $E$-relative branching system induces a representation of the relative Cohn path algebra $C_K^X(E)$.

Fix an $E$-relative branching system $\mathfrak{X}$. Let $M$ be the $K$-module of all functions from $\mathfrak{X}$ taking values in $K$ and let ${\rm Hom}_K(M)$ denote the $K$-algebra of all homomorphisms from $M$ to $M$ (with multiplication given by composition of homomorphisms and the other operations given in the usual way).

Now, for each $e\in E^1$ and for each $v\in E^0$, we will define homomorphisms $S_e$, $S_e^*$ and $P_v$ in ${\rm Hom}_K(M)$. Let
$$S_e \phi=\chi_{R_e}\cdot \phi\circ f_e^{-1},$$ where $\phi \in M$ and $\chi_{R_e}$ is the characteristic function of $R_e$. For $\phi \in M$, we define the homomorphism $S_e^*$ by $$S_e^* \phi=\chi_{D_{r(e)}}\cdot \phi\circ f_e.$$ Finally, for each $v\in E^0$, and for $\phi \in M$, we define $P_v $ by $$P_v \phi=\chi_{D_v} \cdot \phi,$$
that is, $P_v$ is the multiplication operator by $\chi_{D_v}$, the characteristic function of $D_v$.

 \begin{proposition}\label{rep} Let $\mathfrak{X}$ be an $E$-relative branching system. Then there exists a representation (that is, an algebra homomorphism) $\pi: C_K^X(E)\rightarrow {\rm Hom}_K(M)$ such that
$$\pi(e)= S_e, \text{ } \pi(e^*)= S_e^* \text{ and } \pi(v)=P_v,$$ for each $e\in E^1$ and $v\in E^0$.
  \end{proposition}

\demo
Analogous to what is done in \cite[Theorem 2.2]{GR}.
\fim

Let $E$ be a graph, with $E^0$ and $E^1$ countable and $X\subseteq \text{Reg}(E)$. Next we show that there always exists an $E$-relative algebraic branching system in $\R$ associated to $E$ and $X$. We will use the construction we present later, when we build faithful representations of relative Cohn path algebras.
%Our proof is constructive and one can actually obtain a great number of $E$-relative branching systems following the ideas below.

\begin{proposition}
\label{existencebrancsys}
Let $E=(E^0,E^1,r,s)$ be a graph, with $E^0,E^1$ both countable. Then there exists an $E$-relative branching  system $\mathfrak{X}$, where $\mathfrak{X}$ is an (possibly unbounded) interval of $\R$.
\end{proposition}

\begin{proof}
Let $E^0=\{v_i\}_{i \in I}$, where $I=\{1,2,\ldots,N\}$ or $I=\mathbb{N}$. For each $i\geq 1$ define $D_{v_i}=[i-1,i)$.

Our next goal is to define $R_e$ for each $e\in E^1$.

Let $v$ be such that $N:=|s^{-1}(v)|<\infty$. Then $s^{-1}(v)=\{e_1,\ldots, e_N\}$. If $v\in X$ then partition the interval $D_v$ into $N$ intervals, closed on the left and open on the right, and define each $R_{e_1},\ldots,R_{e_N}$ as one of these intervals (with $R_{e_i}\cap R_{e_j}=\emptyset$, for $i\neq j$).
If $v\in Y$ then partition the interval $D_v$ into $N+1$ intervals, closed on the left and open on the right, and define each of $R_{e_1},\ldots,R_{e_N}$ as one of these intervals (with $R_{e_i}\cap R_{e_j}=\emptyset$, for $i\neq j$).

Let $v$ be such that $|s^{-1}(v)|=\infty$. Then
$s^{-1}(v)=\{e_1, e_2,\ldots \}$. Partition the interval $D_v$ into $\infty$ intervals, closed on the left and open on the right, and define each of $R_{e_i}$ as one of these intervals (with $R_{e_i}\cap R_{e_j}=\emptyset$, for $i\neq j$, and length of $R_{e_i}$ equal to $\frac{1}{2^i}$).

Finally, for each $e\in E^1$ define $f_e:D_{r(e)}\rightarrow R_e$ as the affine bijection between these intervals.

It is now standard to check that $\mathfrak{X}$ above is a relative branching system.
\end{proof}

Proposition \ref{existencebrancsys} together with Proposition \ref{rep} guarantees that every relative Cohn path algebra $C_K^X(E)$ of a countable graph $E$  may be represented in ${\rm Hom}_K(M)$. Let us summarize this result in the following corollary:

\begin{corolario}\label{cor1} Given a countable graph $E$, there exists a homomorphism $\pi:C_K^X(E)\rightarrow Hom_K(M)$ such that $$\pi(v)(\phi)=\chi_{D_v}.\phi,\,\,\,\,\,\pi(e)(\phi)=\chi_{R_e}.\phi\circ f_e^{-1}\,\,\,\,\text{ and }\,\,\,\,\pi(e^*)(\phi)=\chi_{D_{r(e)}}.\phi\circ f_e$$ for each $\phi\in M$,  where $M$ is the $K$-module of all functions from $\mathfrak{X}$ taking values in $K$, $\mathfrak{X}$ is an (possible unlimited) interval of $\R$, and $R_e$ and $D_v$ are as in Proposition \ref{existencebrancsys}.
\end{corolario}

\begin{obs}{\rm
Propositions  \ref{rep}, \ref{existencebrancsys}, and Corollary \ref{cor1} generalize to relative Cohn path algebras Theorems 2.2, 3.1, and Corollary 3.2 in \cite{GR}.
}
\end{obs}

Let $(\mathfrak{X},\{R_e\}_{e\in E^1},\{D_v\}_{v\in E^0}, \{f_e\}_{e\in E^1})$ be an $E$-relative branching system. For a closed path $\alpha=e_1...e_n$, let $f_\alpha:D_v\rightarrow R_{e_1}\subseteq D_v$ denote the composition $$f_\alpha:=f_{e_1}\circ...\circ f_{e_n}.$$ Notice that since $\alpha$ is a path $f_\alpha$ is well defined.

We are now in position to prove one of the main results of the paper.

\begin{theorem}\label{faithfulrep} Let $E$ be a graph and $(\mathfrak{X},\{R_e\}_{e\in E^1},\{D_v\}_{v\in E^0}, \{f_e\}_{e\in E^1})$ be an $E$-relative branching system. Let $\pi$ be the representation of $C_K^X(E)$ induced by this $E$-relative branching system. Then $\pi$ is faithful if, and only if, the following conditions are satisfied:
\begin{enumerate}
    \item for each $e\in E^1$ and $v\in E^0$, $R_e$ and $D_v$ are non-empty;
    \item $D_v \neq \bigcup_{e\in s^{-1}(v)} R_e$ for all $v\in Y$; and
    \item for each finite set of paths $\{c^1,...,c^m\}$ in $E$, beginning on the same vertex $w$ with $c=e_1\cdots e_n$ a cycle without exits such that $s(e_i) \notin Y$ for every $i \in \{1,\ldots,n\}$, there is an element $z_0\in D_w$ such that $f_{c}^{j}(z_0)\neq z_0$ for all $j \in \{1,\ldots,m\}$.
\end{enumerate}
\end{theorem}

\begin{proof}

Suppose that Conditions 1. to 3. are satisfied.
Let $\pi$ be a representation of $C_K^X(E)$ and $\phi:C_K^X(E) \rightarrow L_K(E(X)) $ be the isomorphism given in Theorem~\ref{relativetoLPA}. Then $\overline{\pi}:=\pi \circ {\phi}^{-1}$ is a representation of $L_K(E(X))$. We will show that $\overline{\pi}$ is injective, and hence $\pi$ is also injective.

Consider $\alpha \in L_K(E(X))$. By Corollary \ref{Cor:ReductionTheorem} we have that there exist $\mu, \eta \in {\rm Path}(E(X))$ such that either: $0 \neq \mu^*\alpha \eta = ku$, for some $k \in K^{\times}$ and $u \in E(X)^0$; or $0 \neq \mu^*\alpha \eta=p(c)$, where $c$ is a cycle without exits in $E(X)$ and $p(x)$ a polynomial in $K[x,x^{-1}]$.

Suppose the case $0 \neq \mu^* \alpha \eta = ku$. If $u \in E^0\setminus Y$, then $\overline{\pi}(u)=\pi(u)$ and hence for each $\rho\in M$,  where $M$ is the $K$-module of all functions from $\mathfrak{X}$ taking values in $K$, $\overline{\pi}(u)(\rho) = \chi_{D_u} \cdot \rho \neq 0$ by 1.. If $u \in Y$, then $\phi^{-1}(u)=\sum_{e \in s^{-1}(u)} ee^*$ and again for each $\rho\in M$, $\overline{\pi}(u)(\rho)= \chi_{\cup R_e} \cdot \rho \neq 0$ by 1.. Now if $u \in Y'$, we have $\phi^{-1}(u)=u-\sum_{e \in s^{-1}(u)} ee^*$ and for each $\rho\in M$, $\overline{\pi}(u)(\rho)=\chi_{D_v\setminus \cup R_e} \cdot \rho \neq 0$ by 2.. In any case $\overline{\pi}(\mu^*\alpha \eta) \neq 0$ and necessarily $\overline{\pi}(\alpha) \neq 0$.

Assume now that $0 \neq \mu^*\alpha \eta = \sum_{r=m}^n l_r c^r$, for some cycle $c$ without exits in $E(X)$, $l_r \in K$, $m,n \in \mathbb{Z}$, $m \leq n$. If there are negative indices of $c$ in the previous sum, multiply $\mu^*\alpha \eta$ on the left by a certain power of $c$, say $c^s$, so that $c^s \mu^* \alpha \eta = \sum_{r=0}^n l_r c^r$ where $n \in \mathbb{N}$. Notice that if $c=e_1 \cdots e_t$ is a cycle without exits in $E(X)$ then the vertices $s(e_i)_E \notin Y$ for every $i=1,\ldots,t$; hence $\phi^{-1}(c)=c$ and $\overline{\pi}(c^s\mu^* \alpha \eta)=\pi(c^s\mu^*\alpha \eta)$. Then for each $\rho\in M$, $\overline{\pi}(c^s\mu^* \alpha \eta)(\rho)= \sum_{r=0}^n \chi_{D_w} \cdot \rho((f_c)^{-r}) \neq 0$ by 3.. So we have that $\overline{\pi}(c^s\mu^*\alpha \eta) \neq 0 $ which gives $\overline{\pi}(\alpha) \neq 0$.

Since $\overline{\pi}$ is injective, it immediately follows that $\pi$ is injective.

In order to prove the converse statement suppose that one of 1., 2. or 3. is not satisfied. We will show that this implies that $\pi$ is not injective.

For the first situation, if $D_v = \emptyset$ for some vertex $v$ then $\pi(v)=0$; if $R_e = \emptyset$ for some edge $e$ then $\pi(e)=0$.

In the second case imagine $D_v = \bigcup_{e\in s^{-1}(v)} R_e$ for some $v \in Y$. Then $\pi(v-\sum_{e\in s^{-1}(v)}ee^*)=0$.

Finally suppose there exist $j_0$ and a cycle $c=e_1\cdots e_n$ without exits based at $w$ such that $s(e_i) \notin Y$ for every $i$, with the condition that $f_c^{j_0}(z)=z$ for every $z \in D_w$. Then we have that $\pi(c^{j_0})=\pi(w)$ since

$\begin{array}{ll}
f_c(D_{r(e_n)}) & = f_{e_1} \cdots f_{e_n}(D_{r(e_n)}) = f_{e_1} \cdots f_{e_{n-1}}(R_{e_n}) \\
& = f_{e_1} \cdots f_{e_{n-1}}(D_{r(e_{n-1})})= f_{e_1} \cdots f_{e_{n-2}}(R_{e_{n-1}})\\
&= \ldots \\
& = f_{e_1}(R_{e_2}) = R_{e_1} = D_{s(e_1)} = D_w,
\end{array}$

and for each $\rho\in M$, $\pi(c^{j_0})(\rho)= \chi_{f_c^{j_0}(D_{r(e_n)})} \cdot \rho(f_c^{-j_0})= \chi_{D_w} \cdot \rho = \pi(w)(\rho)$.
\end{proof}

\begin{obs}{\rm  We remark that Theorem \ref{faithfulrep} generalizes \cite[Theorem 4.2]{GR} which refers to Leavitt path algebras of row finite graphs with no sinks (and provides only a sufficient condition for faithfulness of the representations). It also generalizes the main result in \cite{GR4} (Theorem 4.3, which deals with separated graphs without loops such that all edges have the same source and the range map is injective, and again only provides a sufficient condition for faithfulness) in the context of non separated graphs.
Therefore, more than providing the correct ambience for the study of relative Cohn path algebras, our theorem above, taking $X$ as the set of regular vertices, improves on the known theory of Leavitt path algebras.}
\end{obs}

% At the same time we prove a more general version of the main result of \cite{GR4} in the context of non separated graphs. Note that \cite[Theorem 4.3]{GR4} is not corollary of the one we obtain in this paper, since in \cite{GR4} separated graphs without loops (where all edges have the same source and the range map is injective) are considered meanwhile here we suppose arbitrary graphs. Further, our general method significantly streamlines the approach that was used in those previous works to prove representation theorems for Leavitt path algebras, since we give necessary and sufficient conditions to guarantee faithfulness and our proof is simpler.}

Finally, motivated by Corollary 4.3 in \cite{GR} and using Theorem \ref{faithfulrep}, we construct below a faithful representation of $C_K^X(E)$ for any graph $E$ and subset $X\subseteq {\rm Reg}(E)$.

\begin{example}\label{Ex:faithfulrepr}{\rm Let $E=(E^0,E^1,r,s)$ be a graph with $E^0, E^1$ both countable, and let $X\subseteq {\rm Reg}(E)$. Consider $D_v$ for $v \in E^0$, and $R_e$ for $e \in E^1$, constructed as in the proof of Proposition \ref{existencebrancsys}.

We need to define bijective maps $f_e:D_{r(e)}\rightarrow R_e$, for each $e\in E^1$. To do this, first fix an irrational number $\theta\in [0,1)$, and let $h_\theta:[0,1)\rightarrow [0,1)$ be defined by $h_\theta(x)=(x+\theta) \ {\rm mod}(1)$, which is a bijective map. Consider for any $e \in E^1$, $g_e:D_{r(e)}\rightarrow [0,1)$ and $\widetilde{g_e}:[0,1)\rightarrow R_e$ as the affine bijections between these intervals respectively. Consider $f_e= \widetilde{g_e} \circ h_{\theta} \circ g_e$. This defines $f_{e}:D_{r(e)}\rightarrow R_{e}$ as a bijective map.

Then we have an $E$-relative branching system $$(\mathfrak{X}, \{D_v\}_{v\in E^0}, \{R_e\}_{e\in E^1}, \{f_e\}_{e\in E^1}),$$ and hence we obtain a representation $\pi:C_K^X(E)\rightarrow {\rm Hom}_K(M)$ (as in Proposition \ref{rep}).

We will prove that the representation $\pi:C_K^X(E)\rightarrow {\rm Hom}_K(M)$ induced by the $E$-relative branching system constructed above is faithful.

All we need to do is verify the hypothesis of Theorem \ref{faithfulrep}. By construction it clearly satisfies 1. and 2.. We check that for each finite set of paths $\{c^1,...,c^m\}$ in $E$, beginning on the same vertex $w$, and with $c=e_1\cdots e_n$ a cycle without exits such that $s(e_i) \notin Y$ for every $i \in \{1,\ldots,n\}$, there is an element $z_0\in D_w$ such that $f_{c}^{j}(z_0)\neq z_0$ for all $j \in \{1,\ldots,m\}$.

So, let $c=e_1\ldots e_n$ be a cycle without exits such that $s(e_i) \notin Y$ for every $i \in \{1,\ldots,n\}$, and beginning on $w$. In this case we have that $R_{e_1}=D_{r(e_{n})}$ and $R_{e_i}=D_{r(e_{i-1})}$ for $i=2,\ldots,n$. Notice that then $\widetilde{g_{e_1}}=g_{e_n}^{-1}$ and $\widetilde{g_{e_i}}=g_{e_{i-1}}^{-1}$ for $i=2,\ldots,n$. Hence
$$f_c = f_{e_1} \circ \ldots \circ f_{e_n} = \widetilde{g_{e_1}} \circ h_{\theta} \circ g_{e_1} \circ \ldots \circ \widetilde{g_{e_n}} \circ h_{\theta} \circ g_{e_n} = \widetilde{g_{e_1}} \circ h_{\theta}^n \circ g_{e_n},$$
and therefore $f_c^j =\widetilde{g_{e_1}} \circ h_{\theta}^{nj} \circ g_{e_n}$. It follows that, if $z\in D_w$ is a rational number, then $f_{e_1}\circ...\circ f_{e_n}(z)$ is a irrational number and hence no rational number is a fixed point for $f_{c}$.
Then, for any finite set $\{c^1,...,c^m\}$ in $E$, beginning on $w$, we may choose $z_0\in D_w$ to be a rational number, and so $f_{c}^{j}(z_0)\neq z_0$ for all $j \in \{1,\ldots,m\}$ as desired.
}
\end{example}

Observe that the construction in Example \ref{Ex:faithfulrepr} gives \cite[Corollary 4.3]{GR} in the case that $E$ is a row finite graph with no sinks and $X={\rm Reg}(E)$ (i.e., for Leavitt path algebras). Also motivated by \cite[Corollary 4.3]{GR}, in \cite{GLR} the authors build faithful representations of graph C*-algebras associated to countable graphs (see Proposition 3.2). Therefore our construction can also be seen as an ``algebraization'' of the C*-construction.

\section*{Acknowledgements}
The first author was partially supported by the Spanish MEC and Fondos FEDER through project MTM2016-76327-C3-1-P; and by the Junta de Andaluc\'{i}a and Fondos FEDER, jointly, through project FQM-7156.

The second author was partially supported by Conselho Nacional de Desenvolvimento Cient\'{i}fico e Tecnol\'{o}gico (CNPq) - Brazil.

\end{document}